 \documentclass[draft]{article}

\usepackage{amsmath,amsfonts,amsthm,amssymb,amscd,cancel,color}
\usepackage{enumitem}
\usepackage{verbatim}
\usepackage[dvipdfmx]{graphicx}
\usepackage{ulem,color}

\theoremstyle{plain}
\newtheorem{theorem}{Theorem}[section]
\newtheorem{lemma}[theorem]{Lemma}
\newtheorem{proposition}[theorem]{Proposition}

\theoremstyle{definition}
\newtheorem{definition}[theorem]{Definition}
\theoremstyle{remark}
\newtheorem{remark}[theorem]{Remark}

\newtheorem{claim}[theorem]{Claim}

\newcommand{\R}{{\mathbb R}}

\def\({\left(}
\def\){\right)}
\def\<{\left\langle}
\def\>{\right\rangle}

\numberwithin{equation}{section}

\begin{document}

\title{A note on the non-existence of small non-trivial compact solutions for Euler-Poisson equation in 1D}

\author{Masaya Maeda\footnote{
		Department of Mathematics and Informatics,
		Graduate School of Science,
		Chiba University,
		Chiba 263-8522, Japan.
		{\it E-mail Address}: {\tt maeda@math.s.chiba-u.ac.jp}},  Tetsu Mizumachi\footnote{
				Division of Mathematical and Information Sciences, Hiroshima University, Kagamiyama 1-7-1, 739-8521 Japan.
				{\it E-mail Address}: {\tt tetsum@hiroshima-u.ac.jp}}}
\maketitle

\begin{abstract}
    In this short note, we prove the non-existence of slow and fast small nontrivial compact solutions for the Euler-Poisson system in $1$D.
    The proof is based on the virial estimate which provides local in space average decay of bounded small solutions.
\end{abstract}

\section{Euler-Poisson equation}
In this note, we consider the Euler-Poisson system for ion dynamics which takes the form
\begin{align}\label{EP}
    \left\{
        \begin{aligned}
            \partial_t{\rho}&=-\partial_x(\rho u)\\
            \partial_t{u}&=-u\partial_x u-\rho^{-1}\partial_x (p(\rho))-\partial_x \phi,\\
            -\partial_x^2 \phi + e^\phi &= \rho,
        \end{aligned}
    \right.
\end{align}
where the known function $p\in C^\infty((0,\infty),\R)$ is the pressure satisfying $p'(s)> 0$ and $p''(s)\geq 0$ for all $s>0$ and
the unknown functions $u:\R^{1+1}\to \R$, $\rho:\R^{1+1} \to (0,\infty)$ and $\phi:\R^{1+1}\to \R$ represent the velocity, density, field and electric potential for ions, respectively.
In the following, we often denote the time derivative $\partial_t$ by $\dot{\ }$ and the spatial derivative $\partial_x$ by ${} '$.

We impose the boundary condition $(u,\rho,\phi)\to (0,1,0)$ as $|x|\to \infty$ and introduce a new unknown $n=\rho-1:\R^{1+1}\to (-1,\infty)$.
Setting $w:(-1,\infty)\to \R$ by $w'(s)=(1+s)^{-1}p'(1+s)$ and $w(0)=0$, we can rewrite \eqref{EP} as
\begin{align}\label{EPw}
    \left\{
        \begin{aligned}
            \dot{n}&=-((1+n) u)'\\
            \dot{u}&=-\(\frac{1}{2}u^2+w(n)+\phi\)',\\
            -\phi'' + q(\phi)&=n,
        \end{aligned}
    \right.
\end{align}
where $q(\phi)=e^\phi-1$.
We further set
$
    k:=p'(1)>0.
$
Then, we have $w'(0)=k$.

We set the energy density and energy as follows.
\begin{align}
    e[n,u,\phi](x)&:=\frac{1}{2}(1+n(x))u(x)^2+W(n(x))+\frac{1}{2}\phi'(x)^2 +R(\phi(x)),\label{eq:energydens}\\
        E[n,u,\phi]&:=\int_{\R}e[n,u,\phi](x)\,dx,\label{def:energy}
\end{align}
where
$
    W(n):=\int_0^n w(s)\,ds$ and
    $R(\phi):=\phi q(\phi)-\int_0^\phi q(s)\,ds.
$
From Lemma \ref{lem:energydensity}, we see that at least formally, $E$ is conserved under the flow of \eqref{EPw}.

We will investigate non-existence of $L^2$-compact solutions in the sense of \cite{dBM04MA}
when mass moves faster or more slower than KdV-like solitary waves of \eqref{EPw}.
\begin{definition}\label{def:L2comp}
  \begin{itemize}
\item[\rm{(i)}]
    We say that $K\subset L^2$ decays uniformly if 
    \begin{align*}
        \forall \epsilon>0,\ \exists R>0,\ \forall f\in K,\ \int_{|x|>R}|f|^2\,dx<\epsilon.
    \end{align*}
\item[\rm{(ii)}]
    We say that a global classical solution $(n,u,\phi)$ of \eqref{EPw} is $L^2$-compact if there exists $y\in C^1(\R,\R)$ such that for all $\epsilon>0$, there exists $R>0$ such that
    \begin{align*}
        \sup_{t\in \R}\int_{|x-y(t)|>R}\(n(t,x)^2+|u(t,x)|^2\)\,dx<\epsilon.
    \end{align*}
  \end{itemize}
\end{definition}


\begin{remark}
A solution $(n,u)$ of \eqref{EPw} is $L^2$-compact if 
$K:=\{(n(t,\cdot+y(t)),u(t,\cdot+y(t)))\ |\ t\in \R\}$ decays uniformly.
Note that we do not assume compactness of  $\overline{K}$.
Indeed, a bounded set $K\subset L^2$ is precompact if and only if $K$ and $\hat{K}$ decay uniformly,
where $\hat{K}:=\{\hat{f}\ |\ f\in K\}$ and $\hat{f}$ is the Fourier transform of $f$,
see \cite{Pego85}.
\end{remark}

Our main results for this note are the following.

\begin{theorem}\label{thm:smallslow}
    For any $\varepsilon>0$, there exists $\delta>0$ such that if $(n,u,\phi)$ is an $L^2$-compact solution satisfying 
    \begin{align*}
        \sup_{t\in \R}\|u\|_{L^2}<\infty,\  \sup_{t\in\R}\|n\|_{L^2\cap L^\infty}<\delta\ \text{and}\ \sup_{t\in\R}|\dot{y}(t)|<k/\sqrt{1+k}-\varepsilon,
    \end{align*}
    and conserves the energy $E$, then $(n,u,\phi)=0$.
\end{theorem}

\begin{theorem}\label{thm:smallfast}
    For any $\varepsilon>0$, there exists $\delta>0$ such that if $(n,u,\phi)$ is a $L^2$-compact solution that satisfies $$\inf_{t\in\R}|\dot{y}(t)|>\sqrt{1+k}+\varepsilon,\  \sup_{t\in \R}\|(n,u)\|_{H^1}<\delta$$ and conserves the energy $E$, then $(u,n,\phi)=0$.
\end{theorem}

A typical example of compact solutions is given by solitary waves, i.e.\ localized traveling-wave solutions.
In dimensions $2$ and $3$, the nonexistence of nontrivial solitary waves for the Euler--Poisson system was proved in \cite{BK24PhysD}.
In the one-dimensional case, the existence of solitary waves with speed $c\in(\sqrt{1+k},c_m)$ for some $c_m>\sqrt{1+k}$ was shown in \cite{BK19JDE}.
Moreover, these solitary waves become small as $c\to \sqrt{1+k}+0$, and therefore Theorem~\ref{thm:smallfast} is optimal in this sense.

For nonlinear dispersive partial differential equations such as KdV-type equations and nonlinear Schr\"odinger equations, the theory of compact solutions plays an important role in the study of asymptotic stability of solitons and scattering for large solutions; see, for example, \cite{MM01ARMA,KM06IM}.
For the related results for stability of solitons for Euler-Poisson equation in 1D, see \cite{BCM2603, BK22ARMA}

To prove Theorems~\ref{thm:smallslow} and \ref{thm:smallfast}, we use virial arguments with suitable virial functionals.
In the regime $|\dot y|\ll1$ and $k>1/8$, Theorem~\ref{thm:smallslow} follows from the standard virial functional
\[
J=\int_{\R}\varphi_A(\cdot-y(t))\,n(t)u(t)\,dx,
\]
where $\varphi_A\sim x$ for $|x|\le A$ and $\varphi_A\sim A$ for $|x|\ge A$.
Since $nu$ is the momentum density, this is essentially the same virial functional as in KdV-type equations; see, for instance, \cite{MM05Nonlinearity}.
However, for general $k>0$, $J$ alone is not sufficient, and we modify the virial functional by adding an additional functional $K$; see \eqref{eq:JandK} and \eqref{def:I} below.
For the nonexistence of fast compact solutions (Theorem~\ref{thm:smallfast}), we use another virial functional that is similar to $J$ but with the momentum density replaced by the energy density.
Such a functional has been used in the study of asymptotic stability of solitons for the FPU lattice in \cite{Mizumachi09CMP}.

The rest of the paper is organized as follows.
In Section~\ref{sec:pre} we collect several identities satisfied by the energy and momentum densities and prove some elementary technical lemmas.
In Section~\ref{sec:prss} we prove Theorem~\ref{thm:smallslow}, and in Section~\ref{sec:prsf} we prove Theorem~\ref{thm:smallfast}.

\section{Preliminary}\label{sec:pre}
We first show that the energy density $e[n,u,\phi]$ is a positive quantity.
\begin{lemma}\label{lem:eispositive}
Let $n,u,\phi:\R\to \R$ and assume $\phi$ is differentiable and $1+n>0$. Then, we have $e[n,u,\phi]\geq 0$.
Furthermore, $e[n,u,\phi]=0$ if and only if $(n,u,\phi)=0$.
\end{lemma}

\begin{proof}
    It suffices to look at each term.
    First, since $1+n>0$, it is trivial that $(1+n) u^2\geq 0$ and $ (1+n(x)) u(x)^2=0$ for all $x\in \R$ if and only if $u=0$.
    Next, we show $W(n)\geq 0$ for all $1+n>0$ and $W(n)=0$ if and only if $n=0$.
    This follows immediately from $W(0)=0$, $W'(0)=w(0)=0$ and $W''(n)=w'(n)=(1+n)^{-1}p'(1+n)>0$.
    Finally, from $R(0)=0$ and $R'(s)=se^{s}$, we see $ \frac{1}{2} \phi'^2+R(\phi)\geq 0$ and $ \frac{1}{2}\phi'(x)^2+R(\phi(x))= 0$ for all $x\in \R$ if and only if $\phi=0$.
\end{proof}

We set the momentum density, which will be important for the virial argument, as follows:
\begin{align}
    m[n,u](x)&:=n(x)u(x),\label{eq:momdensity}
\end{align}
For later use, we will compute conservation laws of energy and momentum for 
small solutions of \eqref{EPw}.
\begin{lemma}\label{lem:energydensity}
Let $(n,u,\phi)$ be a classical solution of \eqref{EPw}. There exists a $\delta>0$
such that if $\|n\|_{L^2}<\delta$,
    \begin{align}\label{eq:enerid}
        -\dot{e}=\(\frac{1}{2} (1+n) u^3 + (1+n) w(n) u +  (1+n)  u \phi+ \phi \partial_x(-\partial_x^2+q'(\phi))^{-1}\partial_x((1+n)u) \)'.
    \end{align}
\end{lemma}

\begin{proof}
    From \eqref{EPw}, we have
    \begin{align*}
        -\dot{e}&
        =-\frac{1}{2}\dot{n}u^2-(1+n) u\dot{u}-w(n)\dot{n}-\partial_t\( \phi'^2-\frac{1}{2}\phi'^2+n\phi-(e^\phi-1-\phi)+\phi\phi''\)\\&
        =\frac{1}{2}( (1+n) u)'u^2+ (1+n) u \(\frac{1}{2}u^2
        +w(n)+\phi\)'+w(n)( (1+n) u)'\\&\quad 
        -(\dot\phi  \phi'+\phi\dot\phi')'+\phi' \dot\phi'+((1+n)u)'\phi -\dot{\phi}\(n-e^\phi+1
        \)\\&
        =\(\frac{1}{2}(1+n)u^3+(1+n)w(n)u+(1+n)u\phi\)' 
        -(\dot\phi \phi'+\phi\dot\phi')'+\phi'\dot\phi'+\dot{\phi}\phi''\\&
         =\(\frac{1}{2}(1+n)u^3+(1+n)w(n)u+(1+n)u\phi-\phi \dot\phi'\)',
    \end{align*}
    where in the first line we have used $W'=w$ and $R(\phi)=\phi e^\phi-e^\phi+1=\phi(n+\phi''+1)-e^\phi+1$ and in the fourth line we used $n-e^\phi+1=-\phi''$.
In view of the proof of Lemma~\ref{lem:nphi}, we can differentiate the equation
$-\phi''+q(\phi)=n$ with respect to $n$ provided $\delta$ is sufficiently small.
    Thus, we have
    \begin{align*}
        \dot{\phi}=\partial_n\phi \dot{n}=-(-\partial_x^2+q'(\phi))^{-1}\partial_x((1+n) u).
    \end{align*}
    Combining the computations, we have \eqref{eq:enerid}.
\end{proof}

\begin{lemma}\label{lem:momdensity}
    Let $(n,u,\phi)$ be a classical solution of \eqref{EPw}.
    Then, we have
    \begin{align}\label{eq:momid}
        -\dot{m}=&\(\(\frac{1}{2}+n\)u^2+S(n)-\frac{1}{2}\phi'^2+Q(\phi)\)',
    \end{align}
    where
    $
        S(n)=nw(n)-W(n)
    $ and $Q(\phi)=\int_0^\phi q(s)\,ds=e^\phi-1-\phi$.
\end{lemma}

\begin{proof}
By \eqref{EPw},
    \begin{align*}
        -\dot{m}&
        =((1+n) u)'u +n \(\frac{1}{2}u^2+w(n)+\phi\)'\\&
        =\frac{1}{2}(u^2)'+ ( n u^2)' +\(\int_0^nsw'(s)\,ds\)'-\frac{1}{2}(\phi'^2)'  +  Q(\phi)'.
    \end{align*}
    Since $\int_0^nsw'(s)\,ds=S(n)$, we have the conclusion.
\end{proof}


Next, we study the properties of $\phi$ when $\|n\|_{L^2}$ is small.
\begin{lemma}\label{lem:nphi}
    There exists $\delta>0$ such that if $\|n\|_{L^2}<\delta$, then $-\phi''+q(\phi)=n$ has a unique solution that satisfies $\|n\|_{L^2}\sim \|\phi\|_{H^2}$.
\end{lemma}

\begin{proof}
    First, we rewrite the equation $-\phi'' +q(\phi)=n$ as
    \begin{align*}
        \phi = \Phi_n[\phi]:=(-\partial_x^2 +1)^{-1}\( n + q(\phi)-\phi\).
    \end{align*}
    Then, for $\delta>0$ chosen later and $n\in B_{L^2}(0,\delta)$, we have
    \begin{align*}
        \|\Phi_n[0]\|_{H^2}=\|n\|_{L^2}<\delta.
    \end{align*}
    Now, set $\tilde{q}(s)=q(s)-s$.
    Then, we have $|\tilde{q}(s_1)-\tilde{q}(s_2)|\lesssim (|s_1|+|s_2|)|s_1-s_2|$ for $s_1,s_2\in [-1,1]$.
    Thus, for $\phi_1,\phi_2\in B_{H^2}(0,2\delta)$, we have
    \begin{align*}
        \|\Phi_n[\phi_1]-\Phi_n[\phi_2]\|_{H^2}=\|\tilde{q}(\phi_1)-\tilde{q}(\phi_2)\|_{L^2}\lesssim \(\|\phi_1\|_{L^\infty}+\|\phi_2\|_{L^\infty}\)\|\phi_1-\phi_2\|_{L^2}.
    \end{align*}
    By the embedding $H^2\hookrightarrow L^\infty$, we see that taking $\delta>0$ sufficiently small, $\Phi_n$ is a contraction mapping on $\overline{B_{H^2}(0,2\delta)}$.
    Furthermore, since $|q(\phi)|\leq 3\phi$ for $\phi$ with $\|\phi\|_{L^\infty}\leq 1$, we have
    \begin{align*}
        \|n\|_{L^2}=\|-\phi''+q(\phi)\|_{L^2}\leq \|\phi''\|_{L^2}+C \|\phi\|_{L^2}.
    \end{align*}
    Similarly, we have
    \begin{align*}
        \|\phi\|_{H^2}=\|n+\tilde{q}(\phi)\|_{L^2}\leq \|n\|_{L^2}+C_\delta \|\phi\|_{H^2}^2\,,
    \end{align*}
where $C_\delta$ is a constant satisfying $\limsup_{\delta\to+0}C_\delta<\infty$.
    Therefore, we have $\|n\|_{L^2}\sim \|\phi\|_{H^2}$.
\end{proof}

When $\|n\|_{L^2}$ is small and $\phi$ solves the third equation of \eqref{EPw}, the uniform decay of $n$ implies uniform decay of $\phi$. 

\begin{lemma}
    There exists $\delta>0$ such that if $n\in L^\infty(\R;L^2)$ decays uniformly and satisfies $\sup_{t\in \R}\|n\|_{L^2}<\delta$, then for any $\epsilon>0$, there exists $R>0$ such that 
    \begin{align*}
        \sup_{t\in\R}\int_{|x|>R}\(|\phi''|^2+|\phi|^2\)\,dx<\epsilon,
    \end{align*}
    where $\phi$ is the solution of $-\phi''+q(\phi)=n$ in $H^2(\R)$.
\end{lemma}

\begin{proof}
    Set $\tilde{\chi}\in C_0^\infty$ such that $1_{[-1,1]}\leq \tilde{\chi}\leq 1_{[-2,2]}$ and $\chi_R:=1-\tilde{\chi}(\cdot/R)$.
    Then, from
    \begin{align*}
 (-\partial_x^2+1)(\chi_R\phi)=\chi_R n - \chi_R(q(\phi)-\phi)-\chi_R''\phi - 2\chi_R'\phi',
    \end{align*}
    we have
    \begin{align*}
        \|\chi_R \phi\|_{H^2}&\lesssim \|\chi_R n\|_{L^2} + \|\phi\|_{L^\infty} \|\chi_R\phi\|_{L^2} + R^{-2} \|\phi\|_{L^2} + R^{-1} \|\phi\|_{H^1}.
    \end{align*}
    Now, for arbitrary $\epsilon>0$, there exists $R_0>0$ such that $\sup_t\|\chi_Rn\|_{L^2}<\epsilon$. From Lemma \ref{lem:nphi}, we have $\|\phi\|_{H^2}\sim \|n\|_{L^2}<\delta$.
Taking $R$ sufficiently large so that $\sup_t R^{-1}\|\phi\|_{H^2}\sim \sup_t R^{-1}\|n\|_{L^2}< R^{-1}\delta<\epsilon$, we have
    \begin{align*}
        \sup_t\|\chi_R\phi\|_{H^2}\lesssim \epsilon + \sup_t\|\phi\|_{L^\infty}\|\chi_R\phi\|_{H^2}.
    \end{align*}
    Finally, taking $\delta>0$ sufficiently small, we have $\sup_t\|\chi_R \phi\|_{H^2}\lesssim \epsilon$.
    This completes the proof.
\end{proof}

When $\|n\|_{L^2\cap L^\infty}$  is small, the energy functional becomes equivalent to the $L^2$ norm of $(n,u)$.

\begin{lemma}\label{lem:energysmall}
    There exists $\delta>0$ such that if $n \in L^\infty(\mathbb{R}; L^2)$ satisfies $\|n\|_{L^2\cap L^\infty}\leq  \delta$ and $\phi$ solves $-\phi''+q(\phi)=n$, then
    \begin{align}\label{eq:smalln2}
        E[n,u,\phi]\sim \int (u^2 + n^2).
    \end{align}
\end{lemma}

\begin{proof}
    It suffices to examine each term of $E$.
    First, since $\|n\|_{L^\infty}$ is small, we have $\int (1+n)u^2\sim \int u^2$ and $W(n)\sim n^2$.
    Furthermore, from Lemma \ref{lem:nphi}, $\|\phi\|_{L^\infty}\lesssim \|\phi\|_{H^2}\sim \|n\|_{L^2}\ll1 $, we see $\phi$ is small in $L^\infty$ so we have $R(\phi)\sim \phi^2$.
    Therefore, we have
    \begin{align*}
        E[n,u,\phi]\sim \int_{\R} \(u^2+n^2+\phi'^2+\phi^2\)\sim \int_{\R}\(u^2+n^2\).
    \end{align*}
\end{proof}

We next prepare several elementary estimates used in the following of this paper.

\begin{lemma}\label{lem:KMMvdB}
Let $A\geq 10$ and $V\in L^\infty$ with $\|V\|_{L^\infty}<1/2$ then, we have
    \begin{align*}
    \|\mathrm{sech}(x/A) (-\partial_x^2  +1+V)^{-1}u\|_{L^2}\leq(1+10A^{-1}) \|(-\partial_x^2  +1+V)^{-1}\mathrm{sech}(x/A) u\|_{L^2}
    \end{align*}
\end{lemma}

\begin{proof}
We follow the proof of Lemma 4.7 in \cite{KMMvdB21AnnPDE}.
    Set 
    \begin{align*}
        h=\mathrm{sech}(x/A) (-\partial_x^2+1+V)^{-1}u,\quad
        k=(-\partial_x^2+1+V)^{-1}\mathrm{sech}(x/A) u.
    \end{align*}
    Then, we have
    \begin{align*}
u=        \mathrm{cosh}(x/A)(-\partial_x^2+1+V)k&=(-\partial_x^2+1+V)\mathrm{cosh}(x/A)h\\&
        =\mathrm{cosh}(x/A) (-\partial_x^2 h+h +Vh-A^{-2}h -2A^{-1}\mathrm{tanh}(x/A)\partial_{x}h)
    \end{align*}
    Thus,
    \begin{align*}
  (-\partial_x^2+1+V)k=-\partial_x^2 h +h+Vh-A^{-2}h -2A^{-1}\mathrm{tanh}(x/A)\partial_{x}h\,.
    \end{align*}
    This implies
    \begin{align*}
        (-\partial_x^2+1+V-A^{-2})h=(-\partial_x^2+1+V)k+2A^{-1}\partial_{x}(\mathrm{tanh}(x/A)h)-2A^{-2}\mathrm{sech}^2(x/A)h,
    \end{align*}
    and thus,
    \begin{align*}
        h=&k+2A^{-1}(-\partial_x^2+1+V)^{-1}\partial_x(\mathrm{tanh}(x/A)h)\\&+A^{-2}(-\partial_x^2+1+V)^{-1}h-2A^{-2}(-\partial_x^2+1+V)^{-1}(\mathrm{sech}^2(x/A)h).\nonumber
    \end{align*}
    Therefore, we have
    \begin{align*}
        \|h\|_{L^2}&
        \leq \|k\|_{L^2} +(4+2A^{-1}+4A^{-1})A^{-1}\|h\|_{L^2},
    \end{align*}
    where we have used $\|(-\partial_x^2+1+V)^{-1}\partial_x^j\|_{L^2\to L^2}\leq 2$ ($j=0,1$) which follows from Neumann series. 
    From $4+6A^{-1}\leq 5$, we have $(1-(4+6A^{-1})A^{-1})^{-1}\leq 1+10A^{-1}$ and 
    we have the conclusion.
\end{proof}

\begin{lemma}\label{lem:techest1}
    There exists $A_0>0$ such that if $A\geq A_0$, then for any $V\in L^\infty$ and $\|V\|_{L^\infty}\leq 1/2$, we have
    \begin{align}
        \sum_{j,k=0}^1\|\mathrm{sech}(x/A)\partial_x^j(-\partial_x^2+1+V)^{-1}\partial_x^k f\|_{ L^2}\lesssim \|\mathrm{sech}(x/A)f\|_{L^2}.\label{eq:techest1}
    \end{align}
\end{lemma}

\begin{proof}
    The term $j=k=0$ reduces to Lemma \ref{lem:KMMvdB}.
    For the term $j=1,k=0$, we have
    \begin{align*}
        \|\mathrm{sech}(x/A)\partial_x(-\partial_x^2+1+V)^{-1}f\|_{L^2}&\leq \|\partial_x(-\partial_x^2+1+V)^{-1}\mathrm{sech}(x/A)f\|_{L^2}
        \\&\quad+\|[\mathrm{sech}(x/A),\partial_x(-\partial_x^2+1+V)^{-1}] f\|_{L^2}.
    \end{align*}
    The first term can be bounded by the right hand side of \eqref{eq:techest1}.
    For the second term,
    \begin{align*}
        &[\mathrm{sech}(x/A),\partial_x(-\partial_x^2+1+V)^{-1}]=\partial_x[\mathrm{sech}(x/A),(-\partial_x^2+1+V)^{-1}]\\&+\quad[\mathrm{sech}(x/A),\partial_x](-\partial_x^2+1+V)^{-1}\\&
        =\partial_x(-\partial_x^2+1+V)^{-1}[\mathrm{sech}(x/A),\partial_x^2](-\partial_x^2+1+V)^{-1}-\frac{1}{A}\mathrm{sech}'(x/A)(-\partial_x^2+1+V)^{-1}\\&
        =-A^{-1}\partial_x(-\partial_x^2+1+V)^{-1}\(A^{-1}\mathrm{sech}''(x/A)+2\mathrm{sech}'(x/A)\partial_x\)(-\partial_x^2+1+V)^{-1}\\&\quad-A^{-1}\mathrm{sech}'(x/A)(-\partial_x^2+1+V)^{-1}.
    \end{align*}
    Thus, using $|\mathrm{sech}'(x)|+|\mathrm{sech}''(x)|\lesssim \mathrm{sech}(x)$, we have
    \begin{align*}
        &\|[\mathrm{sech}(x/A),\partial_x(-\partial_x^2+1+V)^{-1}] f\|_{L^2}\\&\lesssim A^{-1}\(\|\mathrm{sech}(x/A)(-\partial_x^2+1+V)^{-1}f\|+\|\mathrm{sech}(x/A)\partial_x(-\partial_x^2+1+V)^{-1}f\|_{L^2}\)
    \end{align*}
    Thus, we have the desired bound for the term $j=1, k=0$.
    The other terms are similar and therefore we skip the proof.
\end{proof}
Since $(-\partial_x^2+1)^{-1}\delta=\frac12e^{-|x|}$, we have the following.
\begin{lemma}\label{lem:techest2}
    There exists $A_0>0$ such that if $A>A_0$, then we have $$\|\mathrm{\cosh}(x/A)(-\partial_x^2+1)^{-1}\mathrm{sech}(x/A)f\|_{L^2}\lesssim \|f\|_{L^2}.$$
\end{lemma}




\section{Proof of Theorem \ref{thm:smallslow}}\label{sec:prss}

We will prove the non-existence of nontrivial compact solutions by showing the following estimate which naturally arise from the virial argument.
We set $\|u\|_{L^\infty L^2}=\sup_{t\in\R}\|u(t)\|_{L^2}$ and
$\|n\|_{L^\infty(L^2\cap L^\infty)}=\sup_{t\in\R}(\|n(t)\|_{L^2}+\|n(t)\|_{L^\infty})$.
\begin{proposition}\label{prop:condasymptriv}
    There exist $\delta>0$ and $A_0>0$ such that if $(n,u,\phi)$ is a global solution of \eqref{EPw} satisfying energy conservation, $\|u\|_{L^\infty L^2}<\infty$ and $\|n\|_{L^\infty(L^2\cap L^\infty)}\leq \delta$ and if $y\in C^1(\R,\R)$ satisfies $\sup_{t\in\R}|\dot{y}(t)|<\frac{k}{\sqrt{1+k}}$, we have
    \begin{align}\label{avbound}
         \forall A\geq A_0,\ \| e^{-|x-y(t)|/A}(n,u) \|_{L^2L^2}^2\lesssim A\|u\|_{L^\infty L^2}\|n\|_{L^\infty L^2}.
    \end{align}
\end{proposition}


The nonexistence of a small compact solution follows immediately from Proposition \ref{prop:condasymptriv}.

\begin{proof}[Proof of Theorem \ref{thm:smallslow} assuming Proposition \ref{prop:condasymptriv}]
    Assume $(u,n,\phi)$ is a nontrivial compact solution satisfying the assumption of Theorem \ref{thm:smallslow}.
    Then, from Lemma \ref{lem:energysmall}, we have $$0<E[n,u,\phi]\sim \int_{\R}(u(t,x)^2+n(t,x)^2)\,dx,$$
    for arbitrary $t\in \R$.
    Furthermore, for any $c>0$, taking $R>0$ sufficiently large, we have $\int_{|x-y(t)|>R}(u(t,x)^2+n(t,x)^2)<c E[n,u,\phi]$.
    Thus, we have $$0<E[n,u,\phi]\sim \int_{|x-y(t)|\leq R}(u(t,x)^2+n(t,x)^2)\,dx$$
    Taking $A>R$, we have $\int_{|x-y(t)|\leq R}(u(t,x)^2+n(t,x)^2)\,dx\lesssim \|e^{-|x-y(t)|/A}(n,u)\|_{L^2}^2$.
    Thus, we have
    \begin{align*}
        \int_{-T}^T\|e^{-|x-y(t)|/A}(n,u)\|_{L^2}^2\,dt\gtrsim \int_{-T}^T E[n,u,\phi]\,dt=2TE[n,u,\phi]\to \infty,\ T\to\infty.
    \end{align*}
    This contradicts with the bound \eqref{avbound}.
\end{proof}

We prove Proposition \ref{prop:condasymptriv} by virial argument  taking appropriate virial functional.
In the rest of this section, we always assume the assumptions of Proposition \ref{prop:condasymptriv}.

For $A>0$ set
    \begin{align}\label{eq:varphiA}
        \varphi_A(x):=A\mathrm{tanh}(x/A).
    \end{align}
    We note 
    $
        \varphi_A'(x)=\mathrm{sech}^2(x/A),
    $
    and 
    \begin{align}\label{eq:varphiA2}
        |\varphi_A''(x)|\leq 2A^{-1}\varphi_A'(x), \ |\varphi_A'''(x)|\leq 4A^{-2}\varphi_A'(x).
    \end{align}
    We set
    \begin{align}\label{eq:JandK}
        J(t):=\int_{\R}\varphi_A(x-y(t))m[n(t),u(t)](x)\,dx,\ 
        K(t):=-\frac{1}{2}\int_{\R}\varphi_A'(x-y(t))u(t,x)\phi'(t,x)\,dx.
    \end{align}

\begin{lemma}\label{lem:Jdot}
Suppose that $\|u\|_{L^\infty L^2}<\infty$ and $\|n\|_{L^\infty(L^2\cap L^\infty)}<\delta$. Then
    \begin{align*}
        \dot{J}=\frac{1}{2}\int \varphi_A'(x-y)\(u^2+kn^2-\phi'^2+\phi^2\)-\dot{y}\int \varphi_A' nu+R_1,
    \end{align*}
    where the error term $R_1$ satisfying the following estimate:
    \begin{align}\label{eq:errorR}
        |R_1|\lesssim \delta \int_\R \varphi_A'(x-y)\( n^2+u^2+\phi^2\).
    \end{align}
\end{lemma}

\begin{proof}
From Lemma \ref{lem:momdensity}, we have
\begin{align*}
    \dot{J}&=\int_{\R}\varphi_A(x-y)\dot{m}-\dot{y}\int \varphi_A' m\\&
    =\int_{\R}\varphi_A'(x-y)\(\(\frac{1}{2}+n\)u^2+S(n)-\frac{1}{2}\phi'^2+Q(\phi)\)-\dot{y}\int \varphi_A' nu\\&
    =\frac{1}{2}\int \varphi_A'(x-y)\(u^2+kn^2-\phi'^2+\phi^2\)-\dot{y}\int \varphi_A' nu+R_1,
\end{align*}
where
\begin{align*}
    R_1=\int_{\R}\varphi_A'(x-y)\(nu^2+S(n)-\frac{k}{2}n^2+Q(\phi)-\frac{1}{2}\phi^2\).
\end{align*}
Since $\|n\|_{L^2\cap L^\infty}\leq \delta$, we have
\begin{align*}
    |\int_{\R}\varphi_A'(x-y)nu^2|\leq \delta \int_\R \varphi_A'(x-y)u^2.
\end{align*}
Next, recall that $S(0)=S'(0)=0$ and $S''(0)=k$, which implies $|S(n)-\frac{k}{2}n^2|\lesssim |n|^3$.
Therefore, we have
\begin{align*}
    |\int_{\R}\varphi_A'(x-y)\(S(n)-\frac{k}{2}n^2\)|\lesssim \delta \int_\R \varphi_A'(x-y)n^2.
\end{align*}
Finally, from Lemma \ref{lem:nphi}, we have $\|\phi\|_{L^\infty}\lesssim \|\phi\|_{H^2}\sim \|n\|_{L^2}\leq \delta$.
Thus, from $|Q(\phi)-\frac{1}{2}\phi^2|\lesssim |\phi|^3$, we obtain
\begin{align*}
    |\int_{\R}\varphi_A'(x-y)\(Q(\phi)-\frac{1}{2}\phi^2\)|\lesssim \delta \int_\R \varphi_A'(x-y)\phi^2.
\end{align*}
Therefore, we have the conclusion.
\end{proof}

If $k>1/8$ and $|\dot{y}|\ll1$, $J$ is sufficient for virial functional.
However, for general $k>0$, we need additional functional $K$ given in \eqref{eq:JandK}.

\begin{lemma}\label{lem:dotKmod}
Suppose that $\|u\|_{L^\infty L^2}<\infty$ and $\|n\|_{L^\infty(L^2\cap L^\infty)}<\delta$. Then
    \begin{align*}
        \dot{K}=&\frac{1}{2}
        \int\varphi_A'(x-y) \(u\partial_x(-\partial_x^2+1)^{-1}\partial_xu+ k ((\phi)'')^2+(1+k) (\phi')^2\)+R_2,
    \end{align*}
    where $R_2$ satisfies
    \begin{align}\label{eq:errorR2}
        |R_2|\lesssim (\delta+A^{-1})\int_\R \varphi_A'(x-y)\( n^2+u^2+\phi^2+\phi'^2\) .
    \end{align}
\end{lemma}

\begin{proof}
    By \eqref{EPw}, we have
    \begin{align}
        &2\dot{K}=-\int \varphi_A' \dot{u}\phi'-\int \varphi_A' u\dot{\phi}'+\dot{y}\int\varphi_A''u\phi'\label{eq:dotK}\\&
        =\int \varphi_A'\(\frac{1}{2}u^2+w(n)+\phi\)'\phi'
        +\int \varphi_A' u\partial_x (-\partial_x^2+q'(\phi))^{-1}\partial_x((1+n)u)+\dot{y}\int\varphi_A''u\phi'\nonumber\\&
        =k\int\varphi_A'n'\phi'
        +\int \varphi_A' (\phi')^2+\int\varphi_A' u\partial_x(-\partial_x^2+1)^{-1}\partial_xu
        +\frac{1}{2}\int\varphi_A'(u^2)'\phi'+\int \varphi_A'\(w(n)-kn\)'\phi'
        \nonumber\\&\quad
        +\int \varphi_A'u\partial_x\((-\partial_x^2+q'(\phi))-(-\partial_x^2+1)^{-1}\)\partial_xu
        +\int \varphi_A' u\partial_x (-\partial_x^2+q'(\phi))^{-1}\partial_x(nu)+\dot{y}\int\varphi_A''u\phi'\nonumber\\&
        =k\int\varphi_A'n'\phi'
        +\int \varphi_A' (\phi')^2+\int\varphi_A' u\partial_x(-\partial_x^2+1)^{-1}\partial_xu
        +r_1+r_2+r_3+r_4+r_5,\nonumber
    \end{align}
    where $\varphi_A=\varphi_A(x-y)$.
    For $r_1$, by \eqref{eq:varphiA2} and Lemma \ref{lem:nphi}, we have
    \begin{align*}
        |r_1|&\leq \frac{1}{2}|\int \varphi_A'' u^2 \phi'|+\frac{1}{2}\int\varphi_A' u^2|\phi''|
        \lesssim A^{-1}\|\phi'\|_{L^\infty} \int \varphi_A'u^2+\int \varphi_A'u^2 (|n|+q(\phi)) \\&
        \lesssim A^{-1}\|n\|_{L^2} \int \varphi_A'u^2+\|n\|_{L^2\cap L^\infty}\int \varphi_A' u^2
        \lesssim \|n\|_{L^2\cap L^\infty}\int \varphi_A'u^2,
    \end{align*}
    where in the third line we have used $|q(\phi)|\lesssim |\phi|$.
    For $r_2$, by $|w(n)-k n|\lesssim n^2$, we have
    \begin{align*}
        |r_2|&\leq |\int \varphi_A''(w(n)-kn)\phi'|+|\int \varphi_A' (w(n)-kn)(-n+q(\phi))|\\&
        \lesssim A^{-1}\|\phi'\|_{L^\infty}\int \varphi_A' n^2+(\|n\|_{L^\infty}+\|\phi\|_{L^\infty})\int \varphi_A' n^2
        \lesssim \|n\|_{L^2\cap L^\infty}\int \varphi_A'n^2.
    \end{align*}
    For $r_3$, from Lemma \ref{lem:techest1} and $|q'(\phi)-1|\lesssim |\phi|$, we have
    \begin{align*}
        |r_3|&=|\int \varphi_A' u \partial_x (-\partial_x^2+1)^{-1}(q'(\phi)-1)(-\partial_x^2+q'(\phi))^{-1}\partial_xu|
        \\&\leq 
         \|\mathrm{sech}((x-y)/A)u\|_{L^2}
         \|\mathrm{sech}((x-y)/A) \partial_x (-\partial_x^2+1)^{-1}(q'(\phi)-1)(-\partial_x^2+q'(\phi))^{-1}\partial_xu\|_{L^2}\\&
        \lesssim \|n\|_{L^2}\|\mathrm{sech}((x-y)/A)u\|_{L^2}^2.
    \end{align*}
    For $r_4$, from Lemma \ref{lem:techest1}, 
    \begin{align*}
        |r_4|&\lesssim \|\mathrm{sech}((x-y)/A)u\|_{L^2}\|\mathrm{sech}((x-y)/A) \partial_x(-\partial_x+q'(\phi))^{-1}\partial_x(nu)\|_{L^2}\\&
        \lesssim\|n\|_{L^\infty}\|\mathrm{sech}((x-y)/A)u\|_{L^2}^2.
    \end{align*}
    For $r_5$, from \eqref{eq:varphiA2}, we have
    \begin{align*}
        |r_5|\lesssim A^{-1}\(\|\mathrm{sech}((x-y)/A)u\|_{L^2}^2+\|\mathrm{sech}((x-y)/A)\phi'\|_{L^2}^2\).
    \end{align*}
    Finally, for the 1st term in last line of \eqref{eq:dotK}, from $\phi''=-n+\phi+q(\phi)-\phi$ and $|q(\phi)-\phi|\lesssim |\phi|^2$, we have
    \begin{align*}
        \int \varphi_A' n'\phi'&=-\int \varphi_A' n \phi''-\int \varphi_A'' n \phi'\\&
        =\int \varphi_A' (\phi'')^2-\int\varphi_A' \phi \phi'' - \int \varphi_A'(q(\phi)-\phi)\phi''-\int \varphi_A'' n\phi'\\&
        =\int \varphi_A' (\phi'')^2+\int\varphi_A' (\phi' )^2+\int \varphi_A'' \phi \phi'- \int \varphi_A'(q(\phi)-\phi)\phi''-\int \varphi_A'' n\phi'\\&
        =\int \varphi_A' (\phi'')^2+\int \varphi_A' (\phi')^2+r_6+r_7+r_8.
    \end{align*}
    The error terms $r_6,r_7,r_8$ can be bounded as
    \begin{align*}
        |r_6|&\lesssim A^{-1} \|\mathrm{sech}((x-y)/A)\phi\|_{L^2}\|\mathrm{sech}((x-y)/A)\phi'\|_{L^2},\\
        |r_7|&\lesssim \|n\|_{L^2\cap L^\infty} \|\mathrm{sech}((x-y)/A)\phi\|_{L^2}^2,\\
        |r_8|&\lesssim A^{-1}\|\mathrm{sech}((x-y)/A)n\|_{L^2}\|\mathrm{sech}((x-y)/A)\phi'\|_{L^2},
    \end{align*}
    where for $r_7$, we have used $\|\phi''\|_{L^\infty}\leq \|n\|_{L^\infty}+\|q(\phi)\|_{L^\infty}\lesssim \|n\|_{L^\infty}+\|\phi\|_{L^\infty}\lesssim \|n\|_{L^\infty\cap L^2}$.
    Since the error $R_2$ is given by $R_2=\sum_{j=1}^8r_j$, we have the conclusion.
\end{proof}

We now prove Proposition \ref{prop:condasymptriv}.
\begin{proof}[Proof of Proposition \ref{prop:condasymptriv}]
    Set $c:=\sup_{t\in\R}|\dot{y}(t)|$.
    Then, from the assumption we have $c<k/\sqrt{1+k}$.
    Since $c^2/k<k/(1+k)$, we take $\epsilon\in (c^2/k, k/(1+k)) \subset (0,1)$ and set
    \begin{align}\label{def:I}
        I:=J+(1-\epsilon)K.
    \end{align}
    Then, from
    $
        \dot{I}=\dot{J}+(1-\epsilon)\dot{K}
    $
    and Lemmas \ref{lem:Jdot} and \ref{lem:dotKmod}, we have
    \begin{align*}
        \dot{I}&=\frac{1}{2}\int \varphi_A'\(u^2+kn^2-\phi'^2+\phi^2\)-\dot{y}\int \varphi_A' nu\\&
        \quad+\frac{1-\epsilon}{2}
        \int\varphi_A' \(u\partial_x(-\partial_x^2+1)^{-1}\partial_xu+ k ((\phi)'')^2+(1+k) (\phi')^2\)+R\\&
        =\frac{\epsilon}{2}\int \varphi_A' u^2+\frac{1-\epsilon}{2}\int \varphi_A' u(-\partial_x^2+1)^{-1}u-\dot{y}\int \varphi_A' nu\\&\quad+\frac{1}{2}\int \varphi_A'\(kn^2+(1-\epsilon)k(\phi'')^2+\(k-\epsilon(1+k)\)(\phi')^2+\phi^2\)+R,
    \end{align*}
    where $R=R_1+(1-\epsilon)R_2$ and $\varphi_A'=\varphi_A'(x-y)$.
    
    We claim
        \begin{align*}
            \int \varphi_A' u(-\partial_x^2+1)^{-1}u\geq 0.
        \end{align*}
    Indeed, setting $v=(-\partial_x^2+1)^{-1}u$, from $A\geq 1$ and \eqref{eq:varphiA2}, we have
        \begin{align*}
            \int \varphi_A' u(-\partial_x^2+1)^{-1}u&=\int \varphi_A'(-v''+v)v
            =\int \varphi_A'((v')^2+v^2)+\int \varphi_A'' v'v\\&=\int \varphi_A'((v')^2+v^2)-2A^{-1}\int \varphi_A' |v'v|\geq (1-A^{-1})\int \varphi_A'((v')^2+v^2)\geq 0.
        \end{align*}
    Next, notice that we have $c/k<\epsilon/c$.
    Thus, we can take $a\in (c/k,\epsilon/c)$.
    \begin{align*}
        |\dot{y}\int \varphi_A' nu|\leq \frac{ca}{2}\int \varphi_A'u^2+\frac{c}{2a}\int \varphi_A' n^2.
    \end{align*}
    Since $R$ also satisfy the error bound \eqref{eq:errorR2}, we have,
    \begin{align*}
        \dot{I}\geq &\frac{1}{2}\(\epsilon-ca-C(A^{-1}+\delta)\)\int \varphi_A'u^2+\frac{1}{2}(k-\frac{c}{a}-C(A^{-1}+\delta))\int \varphi_A'n^2\\&
        +\frac{1}{2}((1-\epsilon)k-C(A^{-1}+\delta))\int \varphi_A'(\phi'')^2+\frac{1}{2}(k-\epsilon(1+k)-C(A^{-1}+\delta))\int \varphi_A'(\phi')^2\\&+\frac{1}{2}(1-C(A^{-1}+\delta))\int \varphi_A'\phi^2\\
        \gtrsim &\int \varphi_A'u^2+\int \varphi_A'n^2.
    \end{align*}
    From $|I|\lesssim A\|u\|_{L^2}\|n\|_{L^2}$, integrating the above by $t$, we have the conclusion.
\end{proof}

\section{Proof of Theorem \ref{thm:smallfast}}\label{sec:prsf}
As in Theorem \ref{thm:smallslow}, we show Theorem \ref{thm:smallfast} by proving decay estimates similar to those in Proposition \ref{prop:condasymptriv}.

\begin{proposition}\label{prop:fastsmall}
    Let $y\in C^1(\R,\R)$ satisfy $\inf_{t\in \R}|\dot{y}(t)|>\sqrt{1+k}$.
    Then, there exist $\delta>0$ and $A_0>0$ such that if $(n,u,\phi)$ is a classical global solution of \eqref{EPw} satisfying energy conservation and $\sup_{t\in \R}\|(u,n)\|_{L^2 \cap L^\infty}<\delta$, we have
    \begin{align}\label{avbound2}
         \| e^{-|x-y(t)|/A}(n,u) \|_{L^2L^2}^2\lesssim (\inf_{t\in \R}|\dot{y}(t)|-\sqrt{1+k})^{-1}AE[n,u,\phi],
    \end{align}
    for any $A\geq A_0$.
\end{proposition}

Theorem \ref{thm:smallfast} follows immediately from Proposition \ref{prop:fastsmall}.
\begin{proof}[Proof of Theorem \ref{thm:smallfast} assuming Proposition \ref{prop:fastsmall}]
    The argument is the same as the proof of Theorem \ref{thm:smallslow} assuming Proposition \ref{prop:condasymptriv}.
\end{proof}

\begin{proof}[Proof of Proposition \ref{prop:fastsmall}]
Since \eqref{EPw} is invariant under $(t,u)\mapsto (-t,-u)$,
we may assume without loss of generality that $\dot{y}(t)>\sqrt{1+k}$.
    Set
    \begin{align*}
        L(t)=\int \varphi_A(x-y(t))e[n(t),u(t),\phi(t)](x)\,dx.
    \end{align*}
    Notice that since $e\geq 0$ and $|\varphi_A|\leq A$, we have $|L|\leq AE[n,u,\phi]$.
    By Lemma \ref{lem:energydensity}, we have
    \begin{align}
        \dot{L}&=\int \varphi_A\dot{e}-\dot{y}\int \varphi_A' e\nonumber\\
        &=\int  \varphi_A'\(\frac{1}{2} (1+n) u^3 + (1+n) w(n) u +  (1+n)  u \phi+ \phi \((-\partial_x^2+q'(\phi))^{-1}((1+n)u)'\)' \)\nonumber\\&\quad-\dot{y}\int \varphi_A' \(\frac{1}{2}(1+n)u^2+W(n)+\frac{1}{2}\phi'^2 +R(\phi)\)\nonumber\\&
        = -\frac{1}{2}\dot{y} \int \varphi_A' \(u^2+kn^2+\phi'^2+\phi^2\)+R_1+R_2,\label{eq:fastsmall3}
    \end{align}
    where $\varphi_A=\varphi_A(x-y(t))$ and
    \begin{align}
        R_1:=&\int \varphi_A'\(k nu + \phi (-\partial_x^2+1)^{-1}u\)\nonumber\\
        R_2:=&\int\varphi_A'\(\frac{1}{2}(1+n)u^3+(w(n)-kn)u+nw(n)u+nu\phi+\phi \((-\partial_x^2+q'(\phi))^{-1}(nu)'\)'\)\nonumber\\&
        +\int\varphi_A'\phi \(\((-\partial_x^2+q'(\phi))^{-1}-(-\partial_x^2+1)^{-1}\)u'\)'\nonumber
        \\&-\dot{y}\int\varphi_A'\(\frac{1}{2}nu^2+W(n)-\frac{1}{2}kn^2+R(\phi)-\frac{1}{2}\phi^2\).\label{eq:R2}
    \end{align}
    Here, recall $k=p'(1)=w'(0)$.
    \begin{claim}
        We have
        \begin{align}
            |R_2|\lesssim \delta \dot{y}\int \varphi_A' \(u^2+n^2+\phi'^2+\phi^2\).\label{eq:fastsmall9}
        \end{align}
    \end{claim}
    \begin{proof}
        For the last term in the first line of \eqref{eq:R2}, from Lemma \ref{lem:techest1}, we have
        \begin{align*}
            & |\int\varphi_A'\phi \((-\partial_x^2+q'(\phi))^{-1}(nu)'\)'|\\ &\leq \|\mathrm{sech}((x-y)/A)\phi\|_{L^2}\|\mathrm{sech}((x-y)/A)\partial_x(-\partial_x^2+q'(\phi))^{-1}\partial_x(nu)\|_{L^2}\\&
            \lesssim \|\mathrm{sech}((x-y)/A)\phi\|_{L^2}\|\mathrm{sech}((x-y)/A)(nu)\|_{L^2}\\&
            \lesssim \delta \(\|\mathrm{sech}((x-y)/A)\phi\|_{L^2}^2+\|\mathrm{sech}((x-y)/A)u\|_{L^2}^2\).
        \end{align*}
        For the second line of \eqref{eq:R2}, by
    \begin{align*}
        \partial_x\((-\partial_x^2+q'(\phi))^{-1}-(-\partial_x^2+1)^{-1}\)\partial_x=\partial_x(-\partial_x^2+1)^{-1}(1-q'(\phi))(-\partial_x^2+q'(\phi))^{-1}\partial_x,
    \end{align*}
    Lemma \ref{lem:nphi} and Lemma \ref{lem:techest1}, we have
    \begin{align*}
        &|\int\varphi_A'\phi \(\((-\partial_x^2+q'(\phi))^{-1}-(-\partial_x^2+1)^{-1}\)u'\)'|\\&\leq \|\mathrm{sech}((x-y)/A)\phi\|_{L^2}\|\mathrm{sech}((x-y)/A)\partial_x(-\partial_x^2+1)^{-1}(1-q'(\phi))(-\partial_x^2+q'(\phi))^{-1}\partial_xu\|_{L^2}\\&
        \lesssim\|\mathrm{sech}((x-y)/A)\phi\|_{L^2}\|1-q'(\phi)\|_{L^\infty}\|\mathrm{sech}((x-y)/A)(-\partial_x^2+q'(\phi))^{-1}\partial_xu\|_{L^2}\\&\lesssim
        \delta \|\mathrm{sech}((x-y)/A)\phi\|_{L^2} \|\mathrm{sech}((x-y)/A)u\|_{L^2}\\&\lesssim \delta \(\|\mathrm{sech}((x-y)/A)\phi\|_{L^2}^2+\|\mathrm{sech}((x-y)/A)u\|_{L^2}^2\).
    \end{align*}
    The other terms are easily bounded.
    Notice that we have $\dot{y}\geq 1$.
    \end{proof}
    For $R_1$,
    we have
    \begin{align}
    R_1&=
        \int \varphi_A'\(k(-\phi''+\phi)+(-\partial_x^2+1)^{-1}\phi\)u 
        +k\int\varphi_A'(q(\phi)-\phi)u+ \int \([(-\partial_x^2+1)^{-1},\varphi_A'] \phi\) u\nonumber\\&
        =R_{11}
        +R_{12}+R_{13}.\nonumber
    \end{align}
    We can bound $R_{12}$ as
    \begin{align}\label{eq:fastsmall10}
        |R_{12}|\lesssim \delta \(\|\mathrm{sech}((x-y)/A)\phi\|_{L^2}^2+\|\mathrm{sech}((x-y)/A)u\|_{L^2}^2\).
    \end{align}
    For $R_{13}$, since 
    \begin{align*}
        [(-\partial_x^2+1)^{-1},\varphi_A']=(-\partial_x^2+1)^{-1}[-\partial_x^2,\varphi_A'](-\partial_x^2+1)^{-1},
    \end{align*}
    we have
    \begin{align}
        |R_{13}|&\lesssim \|\mathrm{cosh}((x-y)/A)(-\partial_x^2+1)^{-1}
(\varphi_A'''+2\varphi_A''\partial_x)(-\partial_x^2+1)^{-1}\phi \|_{L^2}\|\mathrm{sech}((x-y)/A)u\|_{L^2}\nonumber\\&
        \lesssim A^{-1}\(\|\mathrm{sech}((x-y)/A)\phi\|_{L^2}^2+\|\mathrm{sech}((x-y)/A)u\|_{L^2}^2\),\label{eq:fastsmall11}
    \end{align}
    where we have used Lemmas \ref{lem:KMMvdB} and \ref{lem:techest2} and \eqref{eq:varphiA2}.
    
    For $R_{11}$
    , we have
    \begin{align*}
        |R_{11}|
        \leq \frac{\sqrt{1+k}}{2}\int \varphi_A'u^2
        +\frac{1}{2\sqrt{1+k}}\int \varphi_A' \(k(-\phi''+\phi)+(-\partial_x^2+1)^{-1}\phi\)^2=R_{111}+R_{112}.
    \end{align*}
    For $R_{112}$,
    \begin{align*}
         R_{112}&
        =\frac{1}{2\sqrt{1+k}}\int \varphi_A'\(k^2(-\phi''+\phi)^2+(2k+1)\phi^2\)\\&\quad
        +\frac{k}{\sqrt{1+k}}\int\varphi_A'\((-\phi''+\phi)(-\partial_x^2+1)^{-1}\phi - \phi^2\)
        +\frac{1}{2\sqrt{1+k}}\int \varphi_A'\(\((-\partial_x^2+1)^{-1}\phi\)^2-\phi^2\)\\&
        =R_{1121}+R_{1122}+R_{1123}.
    \end{align*}
    For $R_{1122}$, we have
    \begin{align*}
       & \frac{\sqrt{1+k}}{k}R_{1122}=\<(-\partial_x^2+1)\phi,\varphi_A'(-\partial_x^2+1)^{-1}\phi\>-\<\phi,\varphi_A'\phi\>\\&
        =\<\phi,[-\partial_x^2+1,\varphi_A'](-\partial_x^2+1)^{-1}\phi\>
        =-\<\phi,\varphi_A'''(-\partial_x^2+1)^{-1}\phi\>
-2\<\phi,\varphi_A''(-\partial_x^2+1)^{-1}\phi'\>.
    \end{align*}
    Therefore, from \eqref{eq:varphiA2}, we have
    \begin{align}\label{eq:fastsmall12}
        |R_{1122}|\lesssim A^{-1}\int \varphi_A' \(\phi^2+\phi'^2\).
    \end{align}
    For $R_{1123}$, set $(-\partial_x^2+1)^{-1}\phi=\psi$. Then, from \eqref{eq:varphiA2} and Lemma \ref{lem:KMMvdB}, we have
    \begin{align}
        2\sqrt{1+k}R_{1123}\nonumber
        &
        =\int \varphi_A'(-(\psi'')^2-2(\psi')^2)-2\int \varphi_A'' \psi' \psi \\ &
        \leq \int \varphi_A'''  \psi^2
        \leq 2A^{-2}\| \mathrm{sech}((x-y)/A)(-\partial_x^2+1)^{-1}\phi\|_{L^2}^2\nonumber\\&
        \leq  4A^{-2} \| \mathrm{sech}((x-y)/A)\phi\|_{L^2}^2 .\label{eq:fastsmall13}
    \end{align}
    For $R_{1121}$, by
        \begin{align*}
            \left|\int \varphi_A'(-\phi''+\phi)^2-\int \varphi_A' n^2\right|&\leq 2\int \varphi_A' |(q(\phi)-\phi)n|+\int \varphi_A' (q(\phi)-\phi)^2\nonumber\\&
            \lesssim \delta \int \varphi_A'(n^2+\phi^2),
        \end{align*}
    we have
    \begin{align}
        R_{1121}\leq 
        \frac{1}{2\sqrt{1+k}}\int \varphi_A'\(k^2n^2+(2k+1)\phi^2\)
        +C_k\delta\int \varphi_A' \(n^2+\phi^2\).\label{eq:fastsmall5}
    \end{align}
    We claim
    \begin{align}
       \int \varphi_A' \phi^2\leq \int \varphi_A' n^2 +(\delta+A^{-1})\int \varphi_A' (n^2+\phi^2).\label{eq:fastsmall7}
    \end{align}
    Indeed, by $\phi=(-\partial_x^2+1)^{-1}(n-(q(\phi)-\phi))$ and Lemma \ref{lem:KMMvdB}, we have
    \begin{align*}
        \int \varphi_A'\phi^2&=\|\mathrm{sech}((x-y)/A)(-\partial_x^2+1)^{-1}(n-(q(\phi)-\phi))\|_{L^2}^2\\&\leq (1+10A^{-1}+C\delta)^2\(\|\mathrm{sech}((x-y)/A)n\|_{L^2}^2+ \|\mathrm{sech}((x-y)/A)\phi\|_{L^2}^2\).
    \end{align*}
    Therefore, \eqref{eq:fastsmall7} follows.

    From \eqref{eq:fastsmall5} and \eqref{eq:fastsmall7}, we have
    \begin{align}
        R_{1121}\leq \frac{\sqrt{1+k}}{2}\int \varphi_A'\(kn^2+\phi^2\)\label{eq:fastsmall8}+C_k(A^{-1}+\delta)\int \varphi_A' \(n^2+\phi^2+\phi'^2\).
    \end{align}
    Finally, combining \eqref{eq:fastsmall3}, \eqref{eq:fastsmall9}, \eqref{eq:fastsmall10}, \eqref{eq:fastsmall11}, \eqref{eq:fastsmall12}, \eqref{eq:fastsmall13} and \eqref{eq:fastsmall8}, we obtain
    \begin{align*}
        \dot{L}\leq& -\frac{1}{2}\(\dot{y}-\sqrt{1+k}-\dot{y}\delta\)\int \varphi_A' u^2 -\frac{k}{2}\(\dot{y}-\sqrt{1+k}-C(\delta \dot{y}+A^{-1})\)\int \varphi_A' n^2\\&
        -\frac{1}{2}\(\dot{y}-\sqrt{1+k}-C(\delta \dot{y}+A^{-1})\)\int \varphi_A'\phi^2-\frac{1}{2}\(\dot{y}-C(\delta \dot{y}+A^{-1})\)\int \varphi_A'\phi'^2.
    \end{align*}
    Therefore, taking $A$ sufficiently large and $\delta$ sufficiently small, we have
    \begin{align*}
        \dot{L}\lesssim  -\int \varphi_A' \(u^2+n^2\).
    \end{align*}
    Integrating this, we obtain \eqref{avbound2}.
\end{proof}

\section*{Acknowledgments}
M.M.  was supported by the JSPS KAKENHI Grant Numbers 23H01079 and 24K06792.
T.M was supported by the JSPS KAKENHI Grant Numbers 21K03328 and 24H00185.

\end{document}